\documentclass[11pt]{article}
\usepackage{amsmath}
\usepackage{amssymb}
\usepackage{amsthm}
\usepackage[usenames]{color}
\usepackage{amscd}
\usepackage{indentfirst}
\usepackage[colorlinks=true,linkcolor=webgreen,filecolor=webbrown,
citecolor=webgreen]{hyperref}
\definecolor{webgreen}{rgb}{0,.5,0}
\definecolor{webbrown}{rgb}{.6,0,0}

\def\1{{\bf 1}}

\def\N{{\Bbb N}}
\def\Z{{\Bbb Z}}

\def\divid{\,|\,}

\def\lcm{\operatorname{lcm}}

\newtheorem{corollary}{Corollary}

\newtheorem{proposition}{Proposition}
\newtheorem{remark}{Remark}

\begin{document}

\title{\bf Averages of Ramanujan sums: \\ Note on two papers by E. Alkan}
\author{L\'aszl\'o T\'oth}
\date{}
\maketitle

\centerline{Ramanujan J. {\bf 35} (2014), 149--156}

\begin{abstract} We give a simple proof and a multivariable generalization of an identity due to
E. Alkan concerning a weighted average of
the Ramanujan sums. We deduce identities for other weighted averages of the Ramanujan
sums with weights concerning logarithms, values of arithmetic functions for gcd's, the Gamma function, the Bernoulli
polynomials and binomial coefficients.
\end{abstract}




\section{Introduction}

Let $c_k(j)$ denote the Ramanujan sums defined for  $k\in
\N :=\{1,2,\ldots \}$ and $j\in \Z$ by
\begin{equation*}
c_k(j):= \sum_{\substack{m=1\\ \gcd(m,k)=1}}^k \exp(2\pi imj/k).
\end{equation*}

Other notations used throughout this note are the following:
${\lfloor x \rfloor}$ is the integer part of $x$, $B_m$ ($m\in \N \cup \{0\}$) are
the Bernoulli numbers, $\varphi$ is Euler's totient function,
$\tau(n)$ and $\sigma(n)$ stand for the  number and the sum of the divisors
of $n$, respectively, $\mu$ is the M\"obius function,
$\Lambda$ is the von Mangoldt function, $*$ denotes the Dirichlet convolution of arithmetical functions.
Other notations will be fixed inside the note.

E. Alkan \cite{Alk2012RJ} considered for $r\in \N$ the weighted
average
\begin{equation} \label{def_weighted_average}
S_r(k):= \frac1{k^{r+1}} \sum_{j=1}^k j^r c_k(j),
\end{equation}
being motivated by the use of \eqref{def_weighted_average}
in proving exact formulas for certain mean square averages of
special values of $L$-functions. See \cite{Alk2011}. He proved an
asymptotic formula for $\sum_{k\le x} S_r(k)$ (\cite[Th.\
1]{Alk2012RJ}), based on the following identity.

\begin{proposition} \label{Th_original_id} {\rm (\cite[Eq. 2.19]{Alk2012RJ})} For every $k,r\in \N$,
\begin{equation} \label{original_id}
S_r(k)= \frac{\varphi(k)}{2k} + \frac1{r+1} \sum_{m=0}^{\lfloor r/2
\rfloor} \binom{r+1}{2m} B_{2m} \prod_{p\divid k}
\left(1-\frac1{p^{2m}}\right).
\end{equation}
\end{proposition}

Note that $\prod_{p\divid k} (1-p^{-2m})=J_{2m}(k)k^{-2m}$, where
$J_{2m}$ is the Jordan function of order $2m$. For the proof of
\eqref{original_id} E. Alkan used H\"older's evaluation of the
Ramanujan sums given by
\begin{equation} \label{Holder}
c_k(j)= \frac{\varphi(k)\mu(k/\gcd(k,j))}{\varphi(k/\gcd(k,j))}
\quad (k\in \N, j\in \Z),
\end{equation}
applied the formula
\begin{equation} \label{rel_prime}
\sum_{\substack{j=1\\ \gcd(j,n)=1}}^n j^r = \frac{n^{r+1}}{r+1}
\sum_{m=0}^{\lfloor r/2 \rfloor} \binom{r+1}{2m}
\frac{B_{2m}}{n^{2m}} \prod_{p \divid n} \left(1-p^{2m-1}\right) \quad
(n,r\in \N, n>1)
\end{equation}
(see \cite[Cor.\ 4]{Sin2009}), and then considered the cases $r$
even and $r$ odd, respectively. The same identity
\eqref{original_id} and the same proof were presented by E. Alkan
also in \cite[Proof of Th. 1]{Alk2012IJNT}.

In this note we offer a more simple proof of \eqref{original_id}. Furthermore, we
establish identities for other weighted averages of the Ramanujan
sums with weights concerning logarithms, values of arithmetic functions for gcd's, the Gamma function, the Bernoulli
polynomials and binomial coefficients. It is possible to derive similar formulas for the
corresponding weighted averages of $|c_k(j)|$ and $(c_k(j))^2$, but we will not go into details.
We remark that properties of the polynomials $\sum_{j=0}^{k-1} c_k(j)x^j$ were
investigated by the author in \cite{Tot2010}. We also present a multivariable generalization of the formula
\eqref{original_id} connected to the ``orbicyclic'' arithmetic function, discussed by V.~A.~Liskovets
\cite{Lis2010} and the author \cite{Tot2012}.

\section{Simple proof of Proposition \ref{Th_original_id}}

To derive \eqref{original_id} use the familiar formula (see, e.g., \cite[Prop.\ 10.1.6]{Coh2007}, \cite[Th.\ 271]{HarWri2008}),
\begin{equation} \label{Raman_sum_eval}
c_k(j)= \sum_{d\divid \gcd(k,j)} d\, \mu(k/d) \quad (k\in \N, j\in
\Z).
\end{equation}

We obtain
\begin{equation*}
S_r(k)= \frac1{k^{r+1}} \sum_{j=1}^k j^r \sum_{d\divid \gcd(k,j)} d\,
\mu(k/d) = \frac1{k^{r+1}} \sum_{d\divid k} d^{r+1} \mu(k/d)
\sum_{m=1}^{k/d} m^r
\end{equation*}
\begin{equation*}
= \sum_{d\divid k} \frac{\mu(d)}{d^{r+1}} \sum_{m=1}^d m^r.
\end{equation*}

It is well known that for every $n,r\in \N$ (see, e.g., \cite[Prop.\ 9.2.12]{Coh2007}, \cite[Sect.\ 3.9]{Com1974}),
\begin{equation*}
\sum_{j=1}^n j^r = \frac1{r+1} \sum_{m=0}^r (-1)^m \binom{r+1}{m}
B_m n^{r+1-m}
\end{equation*}
\begin{equation*} \label{sum_r}
=\frac{n^r}{2} + \frac1 {r+1} \sum_{m=0}^{\lfloor r/2 \rfloor}
\binom{r+1}{2m} B_{2m} n^{r+1-2m}.
\end{equation*}

We deduce
\begin{equation*}
S_r(k)= \sum_{d\divid k} \frac{\mu(d)}{d^{r+1}} \left( \frac{d^r}{2} +
\frac1{r+1} \sum_{m=0}^{\lfloor r/2 \rfloor} \binom{r+1}{2m} B_{2m}
d^{r+1-2m} \right)
\end{equation*}
\begin{equation*}
= \frac1{2} \sum_{d\divid k} \frac{\mu(d)}{d} + \frac1{r+1}
\sum_{m=0}^{\lfloor r/2 \rfloor} \binom{r+1}{2m} B_{2m} \sum_{d\divid
k} \frac{\mu(d)}{d^{2m}},
\end{equation*}
giving \eqref{original_id} by using the elementary convolutional identities on
$\varphi(k)$ and $J_{2m}(k)$. \qed

Note that the original proof presented in \cite{Alk2012RJ} and \cite{Alk2012IJNT},
based on the application of \eqref{Holder} and
\eqref{rel_prime} can be shortened. Using that
\begin{equation*}
\sum_{m=0}^{\lfloor r/2 \rfloor} \binom{r+1}{2m} B_{2m}
=\frac{r+1}{2},
\end{equation*}
valid for every $r\in \N \cup \{0\}$ it is not necessary to split the
proof into the cases $r$ even and odd.

\section{Other weighted averages}

\begin{proposition} For every $k\in \N$,
\begin{equation*}
\frac1{k} \sum_{j=1}^k (\log j) c_k(j) = \Lambda(k) + \sum_{d\divid k}
\frac{\mu(d)}{d} \log (d!).
\end{equation*}
\end{proposition}

\begin{proof}
We obtain, using formula \eqref{Raman_sum_eval},
\begin{equation*}
\frac1{k} \sum_{j=1}^k (\log j) c_k(j) = \frac1{k} \sum_{j=1}^k
(\log j) \sum_{d\divid \gcd(k,j)} d\mu(k/d)
\end{equation*}
\begin{equation*}
 = \sum_{d\divid k} (d/k)
\mu(k/d) \sum_{m=1}^{k/d} \log (dm)
= \sum_{d\divid k} \frac{\mu(d)}{d} \sum_{m=1}^{d} \log (mk/d)
\end{equation*}
\begin{equation*}
= \sum_{d\divid k} \mu(d) \log(k/d) +\sum_{d\divid k} \frac{\mu(d)}{d}
\log(d!),
\end{equation*}
where the first sum is $\mu*\log =\Lambda$, in terms of the
Dirichlet convolution, and the proof is complete.
\end{proof}

\begin{proposition} Let $f$ be an arbitrary arithmetic function. Then for
every $k\in \N$,
\begin{equation*}
\sum_{j=1}^k f(\gcd(j,k)) c_k(j) = \varphi(k)(\mu*f)(k).
\end{equation*}
\end{proposition}

\begin{proof} Here it is convenient to use H\"older's formula \eqref{Holder}, although the proof works out
also applying \eqref{Raman_sum_eval} instead. We have
\begin{equation*}
T_f(k): =\sum_{j=1}^k f(\gcd(j,k)) c_k(j) = \sum_{j=1}^k
f(\gcd(j,k))
\frac{\varphi(k)\mu(k/\gcd(k,j))}{\varphi(k/\gcd(k,j))},
\end{equation*}
and grouping the terms according to the values of $d=\gcd(k,j)$ we
deduce
\begin{equation*}
T_f(k)= \varphi(k) \sum_{d\divid k} f(d) \frac{\mu(k/d)}{\varphi(k/d)}
\sum_{\substack{m=1\\ \gcd(m,k/d)=1}}^{k/d} 1
\end{equation*}
\begin{equation*}
=\varphi(k) \sum_{d\divid k} f(d) \mu(k/d)= \varphi(k)(\mu*f)(k).
\end{equation*}
\end{proof}

\begin{corollary} For every $k\in \N$,
\begin{equation*}
\sum_{j=1}^k \gcd(j,k) c_k(j) = (\varphi(k))^2.
\end{equation*}
\begin{equation*}
\sum_{j=1}^k \tau(\gcd(j,k)) c_k(j) = \varphi(k).
\end{equation*}
\begin{equation*}
\sum_{j=1}^k \sigma(\gcd(j,k)) c_k(j) = k\varphi(k).
\end{equation*}
\end{corollary}

Let $\Gamma$ be the Gamma function defined for $x>0$ by
\begin{equation*}
\Gamma(x)= \int_0^{\infty} e^{-t}t^{x-1}\, dt.
\end{equation*}

\begin{proposition} For every $k\in \N$, $k>1$,
\begin{equation*}
\frac1{\varphi(k)} \sum_{j=1}^k (\log \Gamma(j/k))
c_k(j) = \frac1{2}\sum_{p\divid k} \frac{\log p}{p-1} -\frac{\log
2\pi}{2}.
\end{equation*}
\end{proposition}

\begin{proof} It is well known that for every $n\in \N$,
\begin{equation} \label{prod_Gamma}
\prod_{k=1}^n \Gamma(k/n) = \frac{(2\pi)^{(n-1)/2}}{\sqrt{n}},
\end{equation}
which is a consequence of Gauss' multiplication formula, cf., e.g., \cite[Eq.\ (3.10)]{Art1964}, \cite[Prop.\ 9.6.33]{Coh2007}.
We obtain from \eqref{Raman_sum_eval} and \eqref{prod_Gamma} that
\begin{equation*}
\sum_{j=1}^k (\log \Gamma(j/k)) c_k(j) =
\sum_{j=1}^k (\log \Gamma(j/k)) \sum_{d\divid
\gcd(k,j)} d\mu(k/d)
\end{equation*}
\begin{equation*}
= \sum_{d\divid k} d\mu(k/d) \sum_{m=1}^{k/d} \log \Gamma
(md/k) = \sum_{d\divid k} (k/d) \mu(d) \sum_{m=1}^d
\log \Gamma(m/d)
\end{equation*}
\begin{equation*}
= \sum_{d\divid k} (k/d) \mu(d) \log \frac{(2\pi)^{(d-1)/2}}{\sqrt{d}}
\end{equation*}
\begin{equation*}
= \log (2\pi) \sum_{d\divid k} \frac{k}{d} \mu(d)\frac{d-1}{2}-
\frac1{2}\sum_{d\divid k} \frac{k}{d}\mu(d)\log d
\end{equation*}
\begin{equation*}
= \log (2\pi) \left( \frac{k}{2} \sum_{d\divid k}\mu(d) -
\frac1{2}\sum_{d\divid k} \frac{k}{d}\mu(d) \right)
-\frac{k}{2}\sum_{d\divid k} \frac{\mu(d)}{d}\log d,
\end{equation*}
where the first sum is zero for $k>1$ and the second sum is
$\varphi(k)$. Now the use of the identity
\begin{equation*}
\sum_{d\divid k} \frac{\mu(d)}{d}\log d = -\frac{\varphi(k)}{k}
\sum_{p\divid k} \frac{\log p}{p-1},
\end{equation*}
(see, e.g., \cite{Problem_AMM1964}, \cite[Ex.\ 10.8.45]{Coh2007}) completes the proof.
\end{proof}

\begin{proposition} For every $k\in \N$,
\begin{equation} \label{binom}
\frac1{2^k} \sum_{j=0}^k \binom{k}{j} c_k(j) =  \sum_{d\divid k}
\mu(k/d) \sum_{\ell=1}^d (-1)^{\ell k/d} \cos^k (\ell\pi/d).
\end{equation}
\end{proposition}

Note the symmetry property $\binom{k}{j} c_k(j)= \binom{k}{k-j} c_k(k-j)$, valid for every
$0\le j\le k$.

\begin{proof} By using \eqref{Raman_sum_eval} and the formula
\begin{equation*}
\sum_{k=0}^{\lfloor n/r \rfloor}  \binom{n}{kr} = \frac{2^n}{r}
\sum_{j=1}^r \cos^n (j\pi/r) \cos (nj\pi/r) \qquad (n,r\in \N),
\end{equation*}
cf., e.g., \cite[p.\ 84]{Com1974}, we conclude
\begin{equation*}
\sum_{j=0}^k \binom{k}{j} c_k(j) = \sum_{j=0}^k \binom{k}{j} \sum_{d\divid
\gcd(k,j)} d\mu(k/d) =  \sum_{d\divid k} d\mu(k/d) \sum_{m=0}^{k/d}  \binom{k}{dm}
\end{equation*}
\begin{equation*}
= 2^k  \sum_{d\divid k} \mu(k/d) \sum_{\ell=1}^d \cos^k (\ell \pi/d) \cos (k\ell \pi/d),
\end{equation*}
where the last factor is $(-1)^{\ell k/d}$. This gives \eqref{binom}.
\end{proof}

Let $B_m(x)$ ($m\in \N \cup \{0\}$) be the Bernoulli polynomials
defined by the expansion
\begin{equation*}
\frac{te^{xt}}{e^t-1}= \sum_{m=0}^{\infty} B_m(x) \frac{t^m}{m!}.
\end{equation*}

It is well known (see, e.g., \cite[Prop.\ 9.1.3]{Coh2007}) that for
every $k,m\in \N$,
\begin{equation*}
\sum_{j=0}^{k-1} B_m(j/k) = \frac{B_m}{k^{m-1}}.
\end{equation*}

We obtain by \eqref{Raman_sum_eval}, similarly as in the proofs of above the next formula.

\begin{proposition} For every $k,m\in \N$,
\begin{equation*}
\sum_{j=0}^{k-1} B_m(j/k) c_k(j) = \frac{B_m}{k^{m-1}} J_m(k).
\end{equation*}
\end{proposition}

\begin{remark} {\rm The Ramanujan sum $j\mapsto c_k(j)$ is the discrete Fourier transform of the
function $n\mapsto \varrho_k(n)$  given below. By
the formula of the inverse discrete Fourier transform we immediately obtain that
\begin{equation*}
\frac1{k} \sum_{j=1}^k \exp(2\pi ijn/k) c_k(j) =
\begin{cases} 1, & \gcd(k,n)=1 \\ 0, & \gcd(k,n)>1 \end{cases} =: \varrho_k(n),
\end{equation*}
valid for every $k,n\in \N$. We refer to \cite{TotHau2011} for details.
}
\end{remark}

\section{A multivariable generalization}

Let $k_1,\ldots,k_n \in \N$ ($n\in \N$) and let
$k:=\lcm(k_1,\ldots,k_n)$. The function of $n$ variables
\begin{equation*} \label{def_func_E}
E(k_1,\ldots,k_n):= \frac1{k}\sum_{j=1}^k c_{k_1}(j)\cdots
c_{k_n}(j)
\end{equation*}
has combinatorial and topological applications, and was investigated
in the papers of V.~A~Liskovets \cite{Lis2010} and of the author
\cite{Tot2012}. Note that all values of $E(k_1,\ldots,k_n)$ are
nonnegative integers. Furthermore, the function $E$ is
multiplicative as a function of several variables (see \cite{Tot2012}
for this concept). Furthermore, it has the following representation (\cite[Prop.\ 3]{Tot2012}):
\begin{equation*}
E(k_1,\ldots,k_n) = \sum_{d_1\divid k_1, \ldots, d_n\divid k_n}
\frac{d_1\mu(k_1/d_1)\cdots d_n\mu(k_n/d_n)}{\lcm(d_1,\ldots,d_n)}.
\end{equation*}

Consider now the average
\begin{equation*}
S_r(k_1,\ldots,k_n):= \frac1{k^{r+1}} \sum_{j=1}^k j^r c_{k_1}(j)\cdots c_{k_n}(j).
\end{equation*}

\begin{proposition} \label{Th_gener} Let $k_1,\ldots,k_n\in \N$, $k=\lcm(k_1,\ldots,k_n)$
and let $r\in \N$. Then
\begin{equation*}
S_r(k_1,\ldots,k_n) = \frac{\varphi(k_1)\cdots \varphi(k_n)}{2k} + \frac1{r+1}
\sum_{m=0}^{\lfloor r/2 \rfloor} \binom{r+1}{2m}
\frac{B_{2m}}{k^{2m}} g_m(k_1,\ldots,k_n),
\end{equation*}
where
\begin{equation*}
g_m(k_1,\ldots,k_n) = \sum_{d_1\divid k_1, \ldots, d_n\divid k_n}
\frac{d_1\mu(k_1/d_1)\cdots
d_n\mu(k_n/d_n)}{(\lcm(d_1,\ldots,d_n))^{1-2m}}
\end{equation*}
is a multiplicative function in $n$ variables.
\end{proposition}

\begin{proof} Similar to the proof of Proposition \ref{Th_original_id} by using  formula \eqref{Raman_sum_eval} for each of
$c_{k_1}(j), \ldots, c_{k_n}(j)$.
\end{proof}

For $n=1$ Proposition \ref{Th_gener} reduces to Proposition \ref{Th_original_id}.

\begin{corollary} {\rm ($r=1$)} For every $k_1,\ldots,k_n\in \N$, with $k=\lcm(k_1,\ldots,k_n)$,
\begin{equation*}
S_1(k_1,\ldots,k_n) = \frac{\varphi(k_1)\cdots \varphi(k_n)}{2k} + \frac{E(k_1,\ldots,k_n)}{2}.
\end{equation*}
\end{corollary}


\section{Acknowledgement}
The author gratefully acknowledges support from the Austrian Science
Fund (FWF) under the project Nr. M1376-N18.


\medskip

\noindent L\'aszl\'o T\'oth \\
Institute of Mathematics, Universit\"at f\"ur Bodenkultur \\
Gregor Mendel-Stra{\ss}e 33, A-1180 Vienna, Austria \\ and \\
Department of Mathematics, University of P\'ecs \\ Ifj\'us\'ag u. 6,
H-7624 P\'ecs, Hungary \\ E-mail: ltoth@gamma.ttk.pte.hu

\end{document}